\documentclass[dvipdfmx]{amsart}

\usepackage{amsmath,amssymb,amsthm,mathrsfs}
\usepackage[dvipdfmx]{graphicx}
\usepackage{here}
\usepackage[dvipdfmx]{hyperref}
\hypersetup{colorlinks=true}

\theoremstyle{definition}

\newtheorem{theo}{Theorem}
\newtheorem{lemm}[theo]{Lemma}

\newcommand{\F}[1]{
{}_{2}F_{1}\left[ \left. 
\begin{matrix}
\frac{1}{3},  \frac{2}{3} \\
1
\end{matrix}
\right| #1
\right]
}

\renewcommand{\a}{\alpha}

\numberwithin{equation}{section}

\title{Hypergeometric expressions of $L$-values for a Borweins theta product of weight $3$}
\author{Ryojun Ito}
\address{Department of Mathematics and Informatics, Graduate School of Science, Chiba University, 
Yayoicho 1-33, Inage, Chiba, 263-8522 Japan. } 
\email{afua9032@chiba-u.jp} 

\begin{document}

\begin{abstract}
In this paper, we consider a modular form of weight 3, which is a product of the Borweins theta series,  and 
express its $L$-values  at $s=1$,   $2$ and $3$ in terms of special values of Kamp{\'e} de F{\'e}riet hypergeometric functions,  which are two-variable generalization of generalized hypergeometric functions.
\end{abstract}
\maketitle

\section{Introduction and Main Results}
For a modular form $f$ of weight $k$ with $q$-expansion $f(q) = \sum_{n=0}^{\infty} a_{n}q^{n}$ ($q = e^{2\pi i\tau}$,  
$\mathrm{Im}(\tau) > 0$),   its $L$-function $L(f,s)$ is defined by 
\begin{align*}
L(f,  s) = \sum_{n=1}^{\infty} \frac{a_{n}}{n^{s}},  \hspace{5mm} \mathrm{Re}(s) > k+1.
\end{align*}
The function $L(f,  s)$ has meromorphic continuation to the whole complex plane with a possible simple pole at $s= k$ when the Fricke involution image $f^{\sharp}$ of $f$ is also a modular form.  
Furthermore,  if $f^{\sharp}(0) =0$,   then $L(f,s)$ is entire (cf. \cite{shimura}).  
In this paper,  we consider the case when $f(q)$ is a product of 
the Borweins theta series \cite{borweins,  bbg}
\begin{align*}
a(q) &:= \sum_{m,n \in \mathbb{Z}} q^{m^{2} + mn + n^{2}},   \\
b(q) &:= \sum_{m,n \in \mathbb{Z}} \omega^{m-n} q^{m^{2} + mn + n^{2}},   \\
c(q) &:= \sum_{m,n \in \mathbb{Z}} q^{\left( m + \frac{1}{3} \right)^{2} +\left( m + \frac{1}{3} \right) \left( n + \frac{1}{3} \right) + \left( n + \frac{1}{3} \right)^{2}},
\end{align*}
which are modular forms of weight $1$.  Here $\omega$ denotes a primitive cube root of unity.  
These are cubic analogues of the Jacobi theta series and satisfy the cubic identity \cite[(2.3)]{borweins}
\begin{align*}
a^{3}(q) = b^{3}(q) + c^{3}(q).   
\end{align*}

In 2010s,  it was proved that some $L$-values for certain modular forms  can be expressed in terms of  special values of 
generalized hypergeometric functions 
\begin{align*}
{}_{A+1}F_{A} \left[\left.
\begin{matrix}
a_{1},  a_{2},  \dots, a_{A+1} \\
a_{1}^{\prime},  \dots, a_{A}^{\prime}
\end{matrix}
\right| 
x
\right]
:= \sum_{n=0}^{\infty} \frac{ (a_{1})_{n} (a_{2})_{n} \cdots (a_{A+1})_{n}  }{ (a_{1}^{\prime})_{n} \cdots (a_{A}^{\prime})_{n}  } \frac{x^{n}}{(1)_{n}},
\end{align*}
and,  their two-variable generalization,  
Kamp{\'e} de F{\'e}riet hypergeometric functions \cite{appkampe, sk}
\begin{align*}
&F_{A^{\prime};B^{\prime};C^{\prime}}^{A;B;C}\left[\left.
\begin{matrix}
a_{1},  \dots, a_{A}  \\
a_{1}^{\prime} , \dots,  a_{A^{\prime}}^{\prime}
\end{matrix}
; 
\begin{matrix}
b_{1} , \dots,  b_{B} \\
b_{1}^{\prime},  \dots, b_{B^{\prime}}^{\prime}
\end{matrix}
;
\begin{matrix}
c_{1},  \dots, c_{C} \\
c_{1}^{\prime},  \dots, c_{C^{\prime}}^{\prime}
\end{matrix}
\right| x, y
\right] \\
&:= 
\sum_{m,n=0}^{\infty} 
\frac{ \prod_{i=1}^{A} (a_{i})_{m+n}  \prod_{i=1}^{B} (b_{i})_{m} \prod_{i=1}^{C} (c_{i})_{n} }{ \prod_{i=1}^{A^{\prime}} (a_{i}^{\prime})_{m+n} \prod_{i=1}^{B^{\prime}} (b_{i}^{\prime})_{m}  \prod_{i=1}^{C^{\prime}} (c_{i}^{\prime})_{n}} 
\frac{x^{m}y^{n}}{(1)_{m}(1)_{n}},  
\end{align*}
where $a_{i}$,  $a_{i}^{\prime}$,  $b_{i}$,  $b_{i}^{\prime}$,  $c_{i}$,  $c_{i}^{\prime}$ are complex parameters with $a_{i}^{\prime}$,  $b_{i}^{\prime}$,  $c_{i}^{\prime} \not\in \mathbb{Z}_{\leq 0}$,  and  $(a)_{n} := \Gamma (a+n) / \Gamma (a)$ denotes the Pochhammer symbol.
We list some  known cases.  
\begin{enumerate}
\item For some theta products $f(q)$ of weight 2,  Otsubo \cite{otsubo} expressed $L(f,   2)$ in terms of ${}_{3}F_{2}(1)$ via regulators.

\item Rogers \cite{rog} ,  Rogers-Zudilin \cite{rz},  Zudilin \cite{zud} and the author \cite{ito1} expressed $L(f,  2)$ for some theta products $f(q)$ of weight 2 in terms of ${}_{3}F_{2}(1)$ by  an analytic method.  
Furthermore,  for the Jacobi theta product which corresponds to the elliptic curve of conductor 32,  Zudilin \cite{zud} expressed $L(f,  3)$ in terms of ${}_{4}F_{3}(1)$.

\item For some quotients $f(q)$ of the Dedekind eta function $\eta (q) = q^{1/24} \prod_{n=1}^{\infty} (1-q^{n})$ of weight 3 (resp. 4,  5),  Rogers-Wan-Zucker \cite{rwz} expressed $L(f,  2)$ (resp.  $L(f,  3)$,  $L(f,  4)$) in terms of special values of the gamma function or generalized hypergeometric functions by an analytic method.  
The author \cite{ito2} expressed $L(f,  1)$ (hence the values $L(f^{\sharp},  2)$ by the functional equation) 
for some theta products $f(q)$ of weight 3 in terms of ${}_{3}F_{2} (1)$ by the Rogers-Zudilin method.

\item 
Samart \cite{sam} expressed $L(f,    3)$ for some eta quotients $f(q)$ of weight 3 in terms of ${}_{5}F_{4}(1)$ via Mahler measures.  

\item
The author \cite{ito3} expressed $L(f,   3)$ and $L(f,  4)$ for some Jacobi theta products $f(q)$ of weight 3 in terms of $F_{1:B:C}^{1:B+1:C+1} (1,1)$ by the Rogers-Zudilin method.

\item 
For certain binary theta series $f$ of odd weight $k \geqq 3$, Osburn and Straub \cite{os} expressed 
$L(f,  k-1)$ in terms of special values of the gamma function by an analytic method.    

\end{enumerate}

In this paper,  we consider the Borweins theta product of weight 3
\begin{align*}
f(q) := \frac{1}{3} b^{2}(q)c(q^{3}),  
\end{align*}
which satisfies the condition $f^{\sharp}(0) = 0$ (so $L(f,s)$ is entire),  and  express its $L$-values $L(f,  1)$,  $L(f,  2) $ and $L(f,3)$ in terms of special values of Kamp\'e de F\'eriet hypergeometric functions.

The main result is the following.

\begin{theo}
We have the following hypergeometric expressions:
\begin{align}
L(f,  1) &= \frac{1}{27} 
F_{1;1;1}^{1;2;2}\left[\left.
\begin{matrix}
1 \\
2
\end{matrix}
; 
\begin{matrix}
1, \frac{4}{3} \\
2
\end{matrix}
;
\begin{matrix}
\frac{1}{3}, \frac{2}{3} \\
1
\end{matrix}
\right| 1, 1
\right],  \label{main1} \\ 
L(f,  2) &= \frac{4\pi}{81\sqrt{3}} \left(
F_{1;1;1}^{1;2;2} \left[ \left. 
\begin{matrix}
1 \\
2
\end{matrix}
;
\begin{matrix}
1,  \frac{5}{3} \\
2
\end{matrix}
;
\begin{matrix}
\frac{1}{3}, \frac{2}{3} \\
1
\end{matrix}
\right|
1,1
\right] - 
F_{1;1;1}^{1;2;2} \left[ \left. 
\begin{matrix}
1 \\
2
\end{matrix}
;
\begin{matrix}
1,  \frac{4}{3} \\
2
\end{matrix}
;
\begin{matrix}
\frac{1}{3}, \frac{2}{3} \\
1
\end{matrix}
\right|
1,1
\right]
\right),   \label{main2}  \\ 
\begin{split}
L(f,  3) &=  \frac{2\pi^{2}}{27} \left( 
F_{1;1;1}^{1;2;2}\left[\left.
\begin{matrix}
\frac{1}{3} \\
\frac{4}{3}
\end{matrix}
; 
\begin{matrix}
\frac{1}{3},  1 \\
\frac{4}{3}
\end{matrix}
;
\begin{matrix}
\frac{1}{3}, \frac{2}{3} \\
1
\end{matrix}
\right| 1, 1
\right] 
- \frac{1}{4}
 F_{1;1;1}^{1;2;2}\left[\left.
\begin{matrix}
\frac{2}{3} \\
\frac{5}{3}
\end{matrix}
; 
\begin{matrix}
\frac{2}{3},  1 \\
\frac{5}{3}
\end{matrix}
;
\begin{matrix}
\frac{1}{3}, \frac{2}{3} \\
1
\end{matrix}
\right| 1, 1
\right]
\right.  \\
& \quad  \left. 
+ \frac{1}{27} F_{1;2;1}^{1;3;2}\left[\left.
\begin{matrix}
1 \\
2
\end{matrix}
; 
\begin{matrix}
1,  1,  \frac{4}{3} \\
2, 2
\end{matrix}
;
\begin{matrix}
\frac{1}{3}, \frac{2}{3} \\
1
\end{matrix}
\right| 1, 1
\right] 
- \frac{2}{27} F_{1;2;1}^{1;3;2}\left[\left.
\begin{matrix}
1 \\
2
\end{matrix}
; 
\begin{matrix}
1,  1,  \frac{5}{3} \\
2, 2
\end{matrix}
;
\begin{matrix}
\frac{1}{3}, \frac{2}{3} \\
1
\end{matrix}
\right| 1, 1
\right]
\right).
\end{split}  \label{main3}
\end{align}
\end{theo}

Note that the double series $F_{A:B:C}^{A:B+1: C+1} (x,y)$ converges absolutely on $|x| \leq 1$ and $|y| \leq 1$ when the parameters satisfy the three conditions \cite{hms}
\begin{align*}
&\mathrm{Re}\left( \sum_{i=1}^{A} a_{i}^{\prime} + \sum_{i=1}^{B} b_{i}^{\prime} 
- \sum_{i=1}^{A} a_{i} - \sum_{i=1}^{B+1} b_{i}  \right) >0, \\
&\mathrm{Re}\left( \sum_{i=1}^{A} a_{i}^{\prime} + \sum_{i=1}^{C} c_{i}^{\prime} 
- \sum_{i=1}^{A} a_{i} - \sum_{i=1}^{C+1} c_{i}  \right) >0, \\
&\mathrm{Re}\left( \sum_{i=1}^{A} a_{i}^{\prime} + \sum_{i=1}^{B} b_{i}^{\prime} + \sum_{i=1}^{C} c_{i}^{\prime} 
- \sum_{i=1}^{A} a_{i} - \sum_{i=1}^{B+1} b_{i} - \sum_{i=1}^{C+1} c_{i} \right) >0.
\end{align*}

To prove the main result,  we use the Rogers-Zudilin method.   Its strategy is as follows.
We start with the Mellin transformation of $f(q)$:  For $n \in \mathbb{Z}_{\geq 1}$,  
\begin{align}
L(f,  n)  = \frac{(-1)^{n-1}}{3(n-1)!}  \int_{0}^{1} b^{2}(q) c(q^{3}) (\log q)^{n-1} \frac{dq}{q}. \label{mellin}
\end{align}
Set $\a = c^{3}(q) / a^{3}(q)$. Note that  we have $1 -\a = b^{3}(q) / a^{3}(q)$ by the cubic identity.
The key formulas to give a hypergeometric expression of $L(f,n)$ are  the following:  
\begin{align}
a(q) = \F{\a},    \hspace{6mm}
a^{2}(q) \frac{dq}{q} = \frac{d\a}{\a (1-\a)}.   \label{transformation}
\end{align}
The former is \cite[p.97,  (2.26)]{ramanujan5},  and the latter follows from the former and \cite[p.87,  Entry 30]{ramanujan2}.
By these transformation formulas and some computations,  we can reduce \eqref{mellin} to an integral of the form 
\begin{align*}
\int_{0}^{1} P(\a)
{}_{A+1}F_{A}  
\left[\left. 
\begin{matrix}
a_{1},  a_{2} ,  \dots,  a_{A+1} \\
a_{1}^{\prime} ,  \dots,  a_{A}^{\prime}
\end{matrix}
\right| \a
\right]
\F{\a}\frac{d\a}{\a (1-\a)}.  
\end{align*}
Here $P(\a)$ denotes a polynomial in $\a^{k}(1-\a)^{l}$ for various $k$ and $l$.
Then the formulas are obtained by the integral expression 
\begin{align}
\begin{split}
&\frac{\Gamma (a) \Gamma (a^{\prime} - a)}{\Gamma (a^{\prime})}
F_{1 ; B  ;  C}^{1 ; B + 1;C + 1}\left[\left.
\begin{matrix}
a \\
a^{\prime}
\end{matrix}
;
\begin{matrix}
b_{1},  \dots,  b_{B+1} \\
b_{1}^{\prime},  \dots,  b_{B}^{\prime}
\end{matrix}
; 
\begin{matrix}
c_{1},  \dots,  c_{C+1} \\
c_{1}^{\prime},  \dots, c_{C}^{\prime}
\end{matrix}
\right| x, y
\right] \\
&= \int_{0}^{1} t^{a} (1-t)^{a^{\prime}-a} {}_{B+1}F_{B} \left[\left. 
\begin{matrix}
b_{1},  \dots,  b_{B+1} \\
b_{1}^{\prime},  \dots,  b_{B}^{\prime}
\end{matrix}
\right|
xt
\right]
{}_{C+1}F_{C} \left[\left. 
\begin{matrix}
c_{1},  \dots,  c_{C+1} \\
c_{1}^{\prime},  \dots, c_{C}^{\prime}
\end{matrix}
\right|
yt
\right] \frac{dt}{t(1-t)},
\end{split}
\label{integral}
\end{align}
which  easily follows from the series expansion of ${}_{m+1}F_{m}(x)$ and termwise integration.

\section{Proof}

First,  we show \eqref{main1}.    
We have \cite[(2.1)]{bbg}
\begin{align}
c(q^{3}) =\frac{ a(q) - b(q)}{3},    \label{rel1}
\end{align}
hence 
\begin{align*}
L(f,  1) = \frac{1}{3} \int_{0}^{1} b^{2}(q)c(q^{3}) \frac{dq}{q} 
= \frac{1}{9} \int_{0}^{1} b^{2}(q) (a(q) - b(q)) \frac{dq}{q}. 
\end{align*}
By the transformation formulas \eqref{transformation},  the integral above becomes 
\begin{align*}
&\frac{1}{9} \int_{0}^{1} (1 - \a)^{\frac{2}{3}} \left( 1 - (1 - \a)^{\frac{1}{3}} \right) \F{\a} \frac{d\a}{\a(1-\a)} \\
&= \frac{1}{9} \int_{0}^{1} \left( (1-\a)^{-\frac{1}{3}} - 1 \right) \F{\a} \frac{d\a}{\a}.
\end{align*}
If we use  
\begin{align}
(1-x)^{-a} - 1 = ax {}_{2}F_{1} \left[ \left. 
\begin{matrix}
1,  a+1  \\
2
\end{matrix}
\right| x
\right],     \label{geom}
\end{align}
which follows from  $(1-x)^{-a} = {}_{1}F_{0} \left[\left.
\begin{matrix}
a \\
\hspace{1mm}
\end{matrix}
\right| x
\right]$,  then we obtain,  by \eqref{integral}, 
\begin{align*}
&\frac{1}{9} \int_{0}^{1} \left( (1-\a)^{-\frac{1}{3}} - 1 \right) \F{\a} \frac{d\a}{\a} \\
&=  \frac{1}{27} \int_{0}^{1}  \a (1- \a)  {}_{2}F_{1} \left[ \left. 
\begin{matrix}
1,  \frac{4}{3} \\
2
\end{matrix}
\right| 
\a \right] \F{\a} \frac{d\a}{\a (1- \a)}  \\
&=\frac{1}{27} F_{1;1;1}^{1;2;2}\left[\left.
\begin{matrix}
1 \\
2
\end{matrix}
; 
\begin{matrix}
1, \frac{4}{3} \\
2
\end{matrix}
;
\begin{matrix}
\frac{1}{3}, \frac{2}{3} \\
1
\end{matrix}
\right| 1, 1
\right].
\end{align*}

Next,  we show \eqref{main2}.  
By applying \eqref{mellin} to $n=2$ and  changing the variable $q = e^{-2\pi u}$,  we have
\begin{align*}
L(f, 2) =\frac{4\pi^{2}}{3} \int_{0}^{\infty} b^{2}(e^{-2\pi u}) c(e^{-6\pi u}) u du.
\end{align*}
If we use the involution formula
\begin{align}
b(e^{-2\pi u}) = \frac{1}{\sqrt{3} u} c(e^{ - \frac{2\pi}{3u}}),  \label{inv}
\end{align}
which follows from $b(q) = \eta^{3}(q) / \eta (q^{3})$,  $c(q) = 3\eta^{3}(q^{3})/ \eta (q)$ and an involution formula  of $\eta (q)$,     then we obtain 
\begin{align*}
L(f, 2) = \frac{4\pi^{2}}{27\sqrt{3}} \int_{0}^{\infty} c^{2} (e^{-\frac{2\pi}{3u}}) b(e^{- \frac{2\pi }{9 u}}) \frac{du}{u^{2}}. 
\end{align*}
By the variable transformations $u \mapsto 1/ u$,  $q = e^{-2\pi u}$ and $q \mapsto q^{9}$,  the integral above becomes 
\begin{align*}
\frac{2\pi}{3\sqrt{3}} \int_{0}^{1} c^{2} (q^{3}) b(q) \frac{dq}{q}.
\end{align*}
Applying \eqref{rel1} and  the transformation formulas \eqref{transformation},  we obtain 
\begin{align*}
L(f,  2) & = \frac{2\pi}{27\sqrt{3}} \int_{0}^{1} b(q) (a(q) - b(q))^{2} \frac{dq}{q} \\
&= \frac{2\pi}{27\sqrt{3}} \int_{0}^{1} (1 - \a)^{\frac{1}{3}} \left( 1 - (1-\a)^{\frac{1}{3}} \right)^{2} 
\F {\a} \frac{d\a}{\a (1-\a)}  \\
&=\frac{2\pi}{27\sqrt{3}} \int_{0}^{1} \left( (1-\a)^{-\frac{2}{3}} - 2 (1-\a)^{-\frac{1}{3}} + 1 \right) 
\F{\a} \frac{d\a}{\a}  .
\end{align*}
We have
\begin{align*}
(1-\a)^{-\frac{2}{3}} - 2 (1-\a)^{-\frac{1}{3}} + 1
= \frac{2}{3} \a \left( {}_{2}F_{1} \left[ \left.
\begin{matrix}
\frac{5}{3},  1 \\
2
\end{matrix}
\right| \a \right] 
- {}_{2}F_{1} \left[ \left.
\begin{matrix}
\frac{4}{3},  1 \\
2
\end{matrix}
\right| \a \right]  \right),  
\end{align*}
by \eqref{geom},  hence the formula follows from  \eqref{integral}.

Finally,  we prove \eqref{main3}.  
If we apply \eqref{mellin} to $n=3$ and change the variable $q = e^{-2\pi u}$,   we have
\begin{align*}
L(f,  3) &= \frac{4\pi^{3} }{3} \int_{0}^{\infty} b^{2}(e^{-2\pi u}) c(e^{-6\pi u}) u^{2} du   \\
&= \frac{4\pi^{3}}{3\sqrt{3}} \int_{0}^{\infty} b (e^{-2\pi u}) c(e^{-6\pi u}) \cdot c(e^{- \frac{2\pi}{3u}}) u du. 
\end{align*}
Here we used the involution formula \eqref{inv} for the last equality.  
We know the following Lambert series expansions \cite[Theorem 3.19,  (3.36)]{cooper} and \cite[(23)]{rz}:
\begin{align}
c(q) &= 3\sum_{r,  s =1}^{\infty} \chi_{-3} (r) \left( q^{\frac{rs}{3}} - q^{rs} \right),   \notag \\
b(q)c(q^{3}) &= 3 \sum_{n, k= 1}^{\infty} \chi_{-3} (nk) k q^{nk},   \label{wt2}
\end{align}
where $\chi_{-3}$ denotes the primitive Dirichlet character of conductor $3$.  
By these series expansions and  the variable transformation $u \mapsto su/ k$,  the integral above becomes 
\begin{align*}
4 \sqrt{3} \pi^{3} \int_{0}^{\infty} 
\left( \sum_{n,  s=1}^{\infty} \chi_{-3}(n) s^{2} e^{-2\pi u ns} \right)
\left( \sum_{k, r=1}^{\infty} \frac{\chi_{-3}(kr)}{k} \left(e^{-\frac{2\pi kr}{9u}} - e^{-\frac{2\pi kr }{3u}} \right) \right) udu.
\end{align*}
The first series is the Borweins theta product  \cite[Theorem 3.35]{cooper}:
\begin{align*}
c^{3}(q) = 27 \sum_{n,s=1}^{\infty} \chi_{-3}(n)s^{2} q^{ns},   
\end{align*}
which implies  
\begin{align*}
L(f, 3) &= \frac{4\pi^{3}}{9\sqrt{3}} \int_{0}^{\infty} 
c^{3}(e^{-2\pi u}) \sum_{k, r=1}^{\infty} \frac{\chi_{-3}(kr)}{k} \left(e^{-\frac{2\pi kr}{9u}} - e^{-\frac{2\pi kr }{3u}} 
 \right) udu .
\end{align*}
Using \eqref{inv} and  changing the variables $ u \mapsto 1/u$,  $q = e^{-2\pi u}$ and $q \mapsto q^{3}$,  we obtain 
\begin{align*}
L(f, 3) = \frac{2\pi^{2}}{27} \int_{0}^{1} b^{3}(q) 
\sum_{k, r=1}^{\infty} \frac{\chi_{-3}(kr)}{k} \left(q^{\frac{kr}{3}} - q^{kr} \right)  \frac{dq}{q}.
\end{align*}
By Lemma \ref{eisenstein} (below)  and the transformation formulas \eqref{transformation},  the integral above becomes  
\begin{align*}
&\frac{2\pi^{2}}{27} \int_{0}^{1} (1-\a) 
\left( 
\frac{1}{3} \a^{\frac{1}{3}}  {}_{2}F_{1} \left[ \left.
\begin{matrix}
\frac{1}{3},  1\\
\frac{4}{3}
\end{matrix}
\right| \a
\right] - \frac{1}{6} \a^{\frac{2}{3}}  {}_{2}F_{1} \left[ \left.
\begin{matrix}
\frac{2}{3},  1\\
\frac{5}{3}
\end{matrix}
\right| \a
\right] \right.  \\
& \hspace{10mm} \left.   + \frac{\a}{27}   {}_{3}F_{2} \left[ \left.
\begin{matrix}
1, 1,  \frac{4}{3}\\
2, 2
\end{matrix}
\right| \a
\right] - \frac{2\a}{27}   {}_{3}F_{2} \left[ \left.
\begin{matrix}
1, 1,  \frac{5}{3}\\
2, 2
\end{matrix}
\right| \a
\right]
\right) \F{\a } \frac{d \a}{\a (1- \a)} .
\end{align*}
Then the formula follows from \eqref{integral}.
\qed

\begin{lemm}\label{eisenstein}
\begin{align*}
\sum_{k, r=1}^{\infty} \frac{\chi_{-3}(kr)}{k} \left(q^{\frac{kr}{3}} - q^{kr} \right)
=& \frac{1}{3} \a^{\frac{1}{3}}  {}_{2}F_{1} \left[ \left.
\begin{matrix}
\frac{1}{3},  1\\
\frac{4}{3}
\end{matrix}
\right| \a
\right] - \frac{1}{6} \a^{\frac{2}{3}}  {}_{2}F_{1} \left[ \left.
\begin{matrix}
\frac{2}{3},  1\\
\frac{5}{3}
\end{matrix}
\right| \a
\right] \\
& + \frac{\a}{27}   {}_{3}F_{2} \left[ \left.
\begin{matrix}
1, 1,  \frac{4}{3}\\
2, 2
\end{matrix}
\right| \a
\right] - \frac{2\a}{27}   {}_{3}F_{2} \left[ \left.
\begin{matrix}
1, 1,  \frac{5}{3}\\
2, 2
\end{matrix}
\right| \a
\right].
\end{align*}
\end{lemm}
\begin{proof}
Denote the left hand side by $E_{0}(q)$.  Then,  by \eqref{wt2}, 
\begin{align*}
q \frac{d}{dq} E_{0} (q)
= \sum_{k, r =1}^{\infty} \chi_{-3}(kr) r  \left( \frac{1}{3} q^{\frac{kr}{3}} - q^{kr} \right)  
&= \frac{1}{9} b(q^{\frac{1}{3}}) c(q) - \frac{1}{3} b(q) c(q^{3}) \\
&= \frac{1}{9} \left( a(q)c(q) - c^{2}(q) - a(q)b(q) + b^{2}(q) \right) .
\end{align*}
Here,  for the last equality,   we used  \eqref{rel1} and 
\begin{align*}
b(q^{\frac{1}{3}}) = a(q) - c(q),
\end{align*}
which follows from  \cite[Lemma 2.1 (ii),  (iii)]{bbg}.
Hence,  by the transformation formulas  \eqref{transformation},  we have
\begin{align*}
E_{0}(q) &=\frac{1}{9} \int_{0}^{q} \left( a(q)c(q) - c^{2}(q) - a(q)b(q) + b^{2}(q) \right) \frac{dq}{q} \\
&= \frac{1}{9} \int_{0}^{\a} \left( \a^{\frac{1}{3}} - \a^{\frac{2}{3}} - (1-\a)^{\frac{1}{3}} 
+ (1 - \a)^{\frac{2}{3}} \right)  \frac{d\a}{\a (1- \a)} \\
&\left( = 
\frac{1}{9} \int_{0}^{\a} \left( x^{\frac{1}{3}} - x^{\frac{2}{3}} - (1-x)^{\frac{1}{3}} 
+ (1 - x)^{\frac{2}{3}} \right)  \frac{dx}{x (1- x)} \right).
\end{align*}
We divide the integral above into the three integrals
\begin{align}
&\int_{0}^{\a}  x^{\frac{1}{3}} \frac{dx}{x(1-x)},  \label{int1}  \\
&\int_{0}^{\a} x^{\frac{2}{3}} \frac{dx}{x(1-x)},  \label{int2} \\
&\int_{0}^{\a} \left( (1-x)^{\frac{1}{3}} - (1 - x)^{\frac{2}{3}} \right) \frac{dx}{x(1-x)}, \label{int3}
\end{align}
and show that each integral can be written as hypergeometric functions. 
First we compute \eqref{int1}.   If we change the variable $x \mapsto \a x$,  then 
\begin{align*}
\int_{0}^{\a} \frac{x^{\frac{1}{3}}}{x (1-x)} dx 
&= \a^{\frac{1}{3}} \int_{0}^{1} \frac{x^{\frac{1}{3}}}{x(1- \a x)} dx  \\
&= \a^{\frac{1}{3}} \int_{0}^{1} x^{\frac{1}{3}} (1-x) \left( 1 - \a x \right)^{-1} \frac{dx}{x(1-x)} \\
&= \a^{\frac{1}{3}}
\frac{\Gamma \left( \frac{1}{3} \right) \Gamma \left( 1 \right) }{\Gamma \left( \frac{4}{3} \right) }
{}_{2}F_{1} \left[ \left.
\begin{matrix}
\frac{1}{3}, 1 \\
\frac{4}{3}
\end{matrix}
\right| \a \right] 
= 3 \a^{\frac{1}{3}} {}_{2}F_{1} \left[ \left.
\begin{matrix}
\frac{1}{3}, 1 \\
\frac{4}{3}
\end{matrix}
\right| \a \right].
\end{align*}
Here we used the integral expression of generalized hypergeometric functions \cite[p.108,  (4.1.2)]{slater}
\begin{align}
\begin{split}
&\frac{\Gamma (a_{1}) \Gamma (a_{1}^{\prime} - a_{1})}{\Gamma (a_{1}^{\prime})} 
{}_{A+1}F_{A} \left[\left. 
\begin{matrix}
a_{1}, a_{2},  \dots, a_{A+1} \\
a_{1}^{\prime} ,  \dots,  a_{A}^{\prime}
\end{matrix}
\right|
z
\right]  \\
&= \int_{0}^{1} x^{a_{1}} (1-x)^{a_{1}^{\prime} - a_{1}} 
{}_{A}F_{A-1} \left[ \left. 
\begin{matrix}
a_{2},  \dots,  a_{A+1} \\
a_{2}^{\prime},  \dots,  a_{A}^{\prime}
\end{matrix}
\right| zx \right]
\frac{dx}{x(1-x)},
\end{split}  \label{hginterep}
\end{align}
for the last equality.  By similar computations,  one can show that \eqref{int2} coincides with 
\begin{align*}
\frac{3}{2} \a^{\frac{2}{3}} {}_{2}F_{1} \left[ \left.
\begin{matrix}
\frac{2}{3}, 1 \\
\frac{5}{3}
\end{matrix}
\right| \a \right].
\end{align*}
Finally,  by \eqref{geom},  the variable transformation $x \mapsto \a x$ and \eqref{hginterep},   the integral \eqref{int3} becomes 
\begin{align*}
&\int_{0}^{\a} \left( (1-x)^{-\frac{2}{3}} - (1-x)^{-\frac{1}{3}} \right) \frac{dx}{x} \\
&= \int_{0}^{\a} \left( \frac{2}{3} x {}_{2}F_{1} \left[ \left. 
\begin{matrix}
1,  \frac{5}{3} \\
2
\end{matrix}
\right| x \right]
- \frac{1}{3}  x {}_{2}F_{1} \left[ \left. 
\begin{matrix}
1,  \frac{4}{3} \\
2
\end{matrix}
\right| x \right]
\right) \frac{dx}{x} \\
&=\frac{\a}{3}  \int_{0}^{1} x(1-x) 
\left( 2 {}_{2}F_{1} \left[ \left. 
\begin{matrix}
1,  \frac{5}{3} \\
2
\end{matrix}
\right| \a x \right]
-  {}_{2}F_{1} \left[ \left. 
\begin{matrix}
1,  \frac{4}{3} \\
2
\end{matrix}
\right|  \a x \right]
\right)   \frac{dx}{x(1-x)} \\
&= \frac{2\a}{3}  {}_{3}F_{2} \left[ \left. 
\begin{matrix}
1,  1,  \frac{5}{3}\\
2, 2
\end{matrix}
\right| \a  \right]
- \frac{\a}{3} {}_{3}F_{2} \left[ \left. 
\begin{matrix}
1,  1,  \frac{4}{3} \\
2,  2
\end{matrix}
\right|  \a  \right].
\end{align*}
This proves the lemma.
\end{proof}

\section*{Acknowledgment}
I would like to thank Noriyuki Otsubo for helpful comments on a draft version of this paper. 
I also would like to thank Robert Osburn for letting me know the paper \cite{os}.

\end{document}